\RequirePackage{fix-cm}
\documentclass[smallextended]{svjour3} 
\smartqed 
\usepackage{graphics,amsmath,amssymb}
\usepackage{graphicx}
\newtheorem{exam}{Example}
\newtheorem{theoremz}{Main Theorem}

\newtheorem{rem}{Remark}

\begin{document}

\title{A note on the high power Diophantine equations
}


\author{ Mehdi Baghalaghdam \and Farzali Izadi
}


\institute{Mehdi Baghalaghdam \at
Department of Mathematics \\ Faculty of Science \\ Azarbaijan Shahid Madani University \\Tabriz 53751-71379, Iran
\\
\email{mehdi.baghalaghdam@yahoo.com } 
\\
\and
Farzali Izadi \at
Department of Mathematics\\ Faculty of Science \\Urmia University \\Urmia 165-57153, Iran
\\
\email{f.izadi@urmia.ac.ir }
}

\date{Received: date / Accepted: date}
\maketitle
\begin{abstract}
In this paper, we solve the simultaneous Diophantine equations (SDE) $ x_{1}^\mu+ x_{2}^\mu +\cdots+ x_{n}^\mu=k \cdot (y_{1}^\mu+ y_{2}^\mu +\cdots+ y_{\frac{n}{k}}^\mu )$, $\mu=1,3$, where $ n \geq3$, and $k \neq n$, is a divisor of $n$ ($\frac{n}{k}\geq2$), and obtain nontrivial parametric solution for them. Furthermore we present a method for producing another solution for the above Diophantine equation (DE) for the case $\mu=3$, when a solution is given. We work out some examples and find nontrivial parametric solutions for each case in nonzero integers. \\ Also we prove that the other DE $\sum_{i=1}^n p_{i} \cdot x_{i}^{a_i}=\sum_{j=1}^m q_{j} \cdot y_{j}^{b_j}$, has parametric solution and infinitely many solutions in nonzero integers with the condition that: there is a $i$ such that $p_{i}=1$, and\\
$(a_{i},a_{1} \cdot a_{2} \cdots a_{i-1} \cdot a_{i+1} \cdots a_{n} \cdot b_{1} \cdot b_{2} \cdots b_{m})=1$, or there is a $j$ such that
$q_{j}=1$, and $(b_{j},a_{1} \cdots a_{n} \cdot b_{1} \cdots b_{j-1} \cdot b_{j+1} \cdots b_{m})=1$.
Finally we study the DE $x^a+y^b=z^c$.

\end{abstract}
\newpage
\section{Introduction}\label{intro}

\noindent The cubic Diophantine equations has been studied by some mathematicians.\\ Gerardin gave partial solutions of the SDE
\begin{equation}\label{190}
x^k+y^k+z^k=u^k+v^k+w^k; k=1,3
\end{equation}

\noindent in 1915-16 (as quoted by Dickson, pp. 565, 713 of \cite{5})
and additional partial solutions were given by Bremner \cite{1}. Subsequently, complete solutions were given
in terms of cubic polynomials in four variables by Bremner and Brudno \cite{2},
as well as by Labarthe \cite{7}.
\noindent In \cite{4} Choudhry presented a complete four-parameter solution of \eqref{190}
in terms of quadratic polynomials in which each parameter occurs only in the first degree.\\
\noindent In this paper we are interested in the study of the SDE
\begin{equation}\label{301}
x_{1}^\mu+ x_{2}^\mu+ \cdots +x_{n}^\mu=k \cdot (y_{1}^\mu+ y_{2}^\mu + \cdots + y_{\frac{n}{k}}^\mu ), \mu=1,3,
\end{equation}

\noindent where $ n\geq3$, and $k \neq n$, is a divisor of $n$ ($\frac{n}{k}\geq2$).
\\

\noindent Also we study the other DE
\begin{equation}\label{80}
\sum_{i=1}^n p_{i}.x_{i}^{a_i}=\sum_{j=1}^m q_{j}.y_{j}^{b_j}
\end{equation}
where $m, n, a_{i}, b_{i} \in \mathbb{N}$ and $p_{i}, q_{i}\in \mathbb{Z}$.
\noindent This is a generalization of the Fermat's equation.\\

\noindent A common generalization of Fermat's equation is $Ax^a+By^b+Cz^c=0$, where $a$, $b$, $c$ $\in\mathbb{N}
_{\geq2}$ and $A, B, C\in\mathbb{Z}$ are given integers with $ABC\neq0$ and $x, y, z\in\mathbb{Z}$ are variables.\\
\noindent In $1995$, Darmond and Granville (See \cite{6}) proved:\\
If $A, B, C\in\mathbb{Z}$ , $ABC\neq0$ and $a, b, c \in\mathbb{N}_{\geq2}$ be such that $\frac{1}{a}+\frac{1}{b}+\frac{1}{c}<1$, then the equation $Ax^a+By^b+Cz^c=0$ has only finitely many
solutions $x, y, z\in\mathbb{Z}$ with $(x,y,z)=1$.\\
\noindent The following theorem is due to Beukers (See \cite{3}):\\
Let $A, B, C\in\mathbb{Z}$ , $ABC\neq0$ and $a, b, c \in\mathbb{N}_{\geq2}$ such that $\frac{1}{a}+\frac{1}{b}+\frac{1}{c}>1$. Then the equation $Ax^a+By^b+Cz^c=0$ has either zero or infinitely many solutions $x, y, z\in\mathbb{Z}$ with $(x,y,z)=1$.\\

\noindent To the best of our knowledge the SDE \eqref{301} and the DE \eqref{80} has not already been considered by any
other authors.\\

\noindent We prove the following main theorems:

\begin{theoremz}
Let $ n\geq3$, and $k \neq n$, be a divisor of $n$ ($\frac{n}{k}\geq2$).
Then the SDE \eqref{301}
have infinitely many nontrivial parametric solutions in nonzero integers.
\end{theoremz}

\begin{theoremz}

Let $m, n, a_{i}, b_{i} \in \mathbb{N}$ and $p_{i}, q_{i}\in \mathbb{Z}$. Suppose that there is a $i$ such that $p_{i}=1$, and
$(a_{i},a_{1} \cdot a_{2} \cdots a_{i-1} \cdot a_{i+1} \cdots a_{n} \cdot b_{1} \cdot b_{2} \cdots b_{m})=1$
or there is a $j$ such that
$q_{j}=1$, and $(b_{j},a_{1} \cdots a_{n} \cdot b_{1} \cdots b_{j-1} \cdot b_{j+1} \cdots b_{m})=1$.
Then the DE \eqref{80}
has parametric solution and infinitely many solutions in nonzero integers. This solves the DE $x^a+y^b=z^c$ in the special cases of $(a,b,c)=1$, or $(a,b,c)=2$.
\end{theoremz}

\section{The SDE ${\displaystyle \sum_{i=1}^n x_{i} ^\mu=k \cdot \sum_{j=1}^\frac{n}{k} y_{j}^\mu }$; $\mu=1,3$}
In this section, we proof the first main theorem.
\begin{proof}:
\noindent Firstly, it is clear that if\\
\noindent $X=(x_1, \cdots ,x_{n},y_1, \cdots ,y_{\frac{n}{k}})$, and
\noindent $ Y=(X_1, \cdots ,X_{n},Y_1, \cdots ,Y_{\frac{n}{k}}), $\\
\noindent be two solutions for the SDE \eqref{301}, then for any arbitrary rational number $t$, $X+tY$ is also a solution for $\mu=1$, i.e.,\\
\noindent $(x_1+tX_1)+(x_2+tX_2)+ \cdots +(x_{n}+tX_{n})=\\
k \cdot [(y_1+tY_1)+(y_2+tY_2)+ \cdots +(y_{\frac{n}{k}}+tY_{\frac{n}{k}})]$.\\

\noindent We say that $X$ is a trivial parametric solution of the SDE \eqref{301}, if it is nonzero and satisfies the system trivially.
Let $X$, and $Y$ be two proper trivial parametric solutions of the SDE \eqref{301}, (we will introduce them later).

\noindent We suppose that $ X+tY$, is also a solution for the case $\mu=3$, where $t$ is a parameter, we wish to find $t$.
\noindent By plugging\\
\noindent $X+tY=(x_1+tX_1,x_2+tX_2, \cdots ,x_{n}+tX_{n},y_1+tY_1, y_2+tY_2, \cdots ,y_{\frac{n}{k}}+tY_{\frac{n}{k}})$,\\
\noindent into the SDE \eqref{301}, we get:\\
\noindent $(x_1+tX_1)^3+(x_2+tX_2)^3+ \cdots +(x_{n}+tX_{n})^3=\\
k \cdot [(y_1+tY_1)^3+(y_2+tY_2)^3+ \cdots +(y_{\frac{n}{k}}+tY_{\frac{n}{k}})^3]$.
\\

\noindent Since $X$ and $Y$ are solutions for the SDE \eqref{301}, after some simplifications, we obtain:\\
\noindent $3t^2(x_1X_1^2+x_2X_2^2+ \cdots +x_{n}X_{n}^2-ky_1Y_1^2-ky_2Y_2^2- \cdots -ky_{\frac{n}{k}}Y_{\frac{n}{k}}^2)+$\\
\noindent $3t(x_1^2X_1+x_2^2X_2+ \cdots +x_{n}^2X_{n}-ky_1^2Y_1-ky_2^2Y_2- \cdots -ky_{\frac{n}{k}}^2Y_{\frac{n}{k}})=0.$\\

\noindent Therefore $t=0$ or\\
$$t=\frac{x_1^2X_1+ \cdots +x_{n}^2X_{n}-ky_1^2Y_1- \cdots -ky_{\frac{n}{k}}^2Y_{\frac{n}{k}}}{-x_1X_1^2- \cdots -x_{n}X_{n}^2+ky_1Y_1^2+ \cdots +ky_{\frac{n}{k}}Y_{\frac{n}{k}}^2}:=\frac{A}{B}.$$\\
\noindent By substituting $t$ in the above expressions, and clearing the denominator $B^3$, we get an integer solution for the SDE \eqref{301} as follows:\\
\noindent $(x'_1,x'_2, \cdots ,x'_{n},y'_1,y'_2, \cdots ,y'_{\frac{n}{k}})=$\\

\noindent $(x_1B+AX_1,x_2B+AX_2, \cdots ,x_{n}B+AX_{n},y_1B+AY_1,y_2B+AY_2, \cdots ,y_{\frac{n}{k}}B+AY_{\frac{n}{k}})$.\\

\noindent If we pick up trivial parametric solutions $X$ and $Y$ properly, we will get a nontrivial parametric solution for the SDE \eqref{301}.
\noindent We should mention that not every trivial parametric solutions of $X$ and $Y$ necessarily give rise to a nontrivial parametric solution for the SDE \eqref{301}. So the trivial parametric solutions must be chosen properly.\\
\noindent Now we introduce the proper trivial parametric solutions.\\
\noindent For the sake of simplicity, we only write down the trivial parametric solutions for the left hand side of the SDE \eqref{301}. It is clear that the trivial parametric solutions for the right hand side of the SDE \eqref{301} can be similarly found by the given trivial parametric solutions of the left hand side of the SDE \eqref{301}.

\noindent Let $p_i$, $q_i$, $s_i$, $r_i$$\in\mathbb{Z}$.\\

\noindent A) The case $n\geq3$ and $\frac{n}{k}\geq3$:\\

\noindent There are $4$ different possible cases for $n$:\\

\noindent $1$. $n=2\alpha+1$, $\alpha$ is even.\\
\noindent $X_{left}=(x_1,x_2, \cdots ,x_{n})=
(p_1,-p_1,p_2,-p_2, \cdots ,p_\alpha,-p_\alpha, 0)$, and\\
\noindent $Y_{left}=(X_1,X_2, \cdots ,X_{n}) =$
\\
\noindent $(r_1,r_2,-r_1,-r_2,r_3,r_4,-r_3,-r_4, \cdots ,r_{\alpha-1},r_\alpha,-r_{\alpha-1},0,-r_\alpha).$
\\

\noindent $2$. $n=2\alpha+1$, $\alpha$ is odd.\\
\noindent $X_{left}=(p_1,-p_1,p_2,-p_2, \cdots ,p_\alpha,-p_\alpha,0)$, and\\
\noindent $Y_{left}=(r_1,r_2,-r_1,-r_2, \cdots ,- r_{\alpha-2},-r_{\alpha-1},r_\alpha,0,-r_\alpha)$.\\

\noindent $3$. $n=2\alpha$, $\alpha$ is even.\\
\noindent $X_{left}=(p_1,-p_1,p_2,-p_2, \cdots ,p_\alpha,-p_\alpha)$, and\\
\noindent $Y_{left}=(r_1,r_2,-r_1,-r_2, \cdots , r_{\alpha-1},r_\alpha,-r_{\alpha-1},-r_\alpha).$\\

\noindent $4$. $n=2\alpha$, $\alpha$ is odd.\\
\noindent $X_{left}=(p_1,-p_1,p_2,-p_2, \cdots ,p_\alpha,-p_\alpha)$, and\\
\noindent $Y_{left}=(r_1,r_2,-r_1,-r_2, \cdots , r_{\alpha-2},r_{\alpha-1},-r_{\alpha-2},r_\alpha,-r_{\alpha-1},-r_\alpha).$\\

\noindent B) The case $n\geq3$ and $\frac{n}{k}=2$:\\

\noindent $5$. $n=2k\geq3$ and $\frac{n}{k}=2$.\\
\noindent $X=(p_1,-p_1,p_2,-p_2, \cdots ,p_{k},-p_{k},q_1
,-q_1)$,\\
\noindent $Y=(r_1,r_2,r_1,r_2, \cdots ,r_1,r_2)$.\\

\noindent (In the final case we present the trivial parametric solutions X, and Y, completely.)\\

\noindent Finally, it can be easily shown that for every $i \neq j$, we have:\\
\noindent $x_iB+AX_i$ $\neq\pm$ $(x_jB+AX_j)$, $y_iB+AY_i$ $\neq\pm$ $(y_jB+AY_j)$ and\\ $x_iB+AX_i$$\neq\pm$ $(y_jB+AY_j)$, \\
\noindent i.e., it is really a nontrivial parametric solution in terms of the parameters $p_i$, $q_i$, $r_i$ and $s_i$.\\
\noindent It is clear that by fixing all of the parameters $p_i$, $q_i$, $r_i$ and $s_i$, but one parameter, we can obtain a nontrivial one parameter parametric solution for the SDE \eqref{301} and by changing properly the fixed values, finally get infinitely many nontrivial one parameter parametric solutions for the SDE \eqref{301}.

\noindent Now, the proof of the first main theorem is completed. $\spadesuit$\\
\end{proof}
\noindent Now, we work out some examples.

\begin{exam} $a^\mu+b^\mu+c^\mu+d^\mu+e^\mu+f^\mu=3 \cdot (g^\mu+h^\mu)$; $\mu=1,3$.\\

\noindent Trivial parametric solutions (case $5$):\\
\noindent $X=(x_1,x_2,x_3,x_4,x_5,x_6,x_7,x_8)=(p_1,-p_1,p_2,-p_2,p_3,-p_3,q_1,-q_1),$\\
\noindent $Y=(y_1,y_2,y_3,y_4,y_5,y_6,y_7,y_8)=(r_1,r_2,r_1,r_2,r_1,r_2,r_1,r_2)$,\\
\noindent $A=x_1^2y_1+x_2^2y_2+x_3^2y_3+x_4^2y_4+x_5^2y_5+x_6^2y_6-3x_7^2y_7-3x_8^2y_8$,\\
\noindent $B=-x_1y_1^2-x_2y_2^2-x_3y_3^2-x_4y_4^2-x_5y_5^2-x_6y_6^2+3x_7y_7^2+3x_8y_8^2$,\\

\noindent nontrivial parametric solution:\\
\noindent $a=x_1B+Ay_1=p_1B+r_1A$,\\
\noindent $b=x_2B+Ay_2=-p_1B+r_2A$,\\
\noindent $c=x_3B+Ay_3=p_2B+r_1A$,\\
\noindent $d=x_4B+Ay_4=-p_2B+r_2A$,\\
\noindent $e=x_5B+Ay_5=p_3B+r_1A$,\\
\noindent $f=x_6B+Ay_6=-p_3B+r_2A$,\\
\noindent $g=x_7B+Ay_7=q_1B+r_1A$,\\
\noindent $h=x_8B+Ay_8=-q_1B+r_2A$.\\

\noindent Example $1$:\\
\noindent $p_1=2$, $p_2=3$, $p_3=4$, $q_1=5$, $r_1=1$, $r_2=6$,\\
\noindent Solution:\\
\noindent $371^\mu+756^\mu+476^\mu+651^\mu+581^\mu+546^\mu=3 \cdot (686^\mu+441^\mu)$; $\mu=1,3$..\\

\noindent Example $2$:\\
\noindent $p_1=7$, $p_2=9$, $p_3=5$, $q_1=2$, $r_1=4$, $r_2=1$,\\
\noindent Solution:\\
\noindent $257^\mu+458^\mu+167^\mu+548^\mu+347^\mu+368^\mu=3 \cdot (482^\mu+233^\mu)$; $\mu=1,3$.
\end{exam}

\begin{exam}
$a^\mu+b^\mu+c^\mu+d^\mu+e^\mu+f^\mu+g^\mu+h^\mu=2 \cdot (i^\mu+j^\mu+k^\mu+l^\mu)$; $\mu=1,3$.\\

\noindent Trivial parametric solutions (case $3$):\\

\noindent $X=(x_1,x_2,\cdots,x_{12})=$
\noindent $(p_1,-p_1,p_2,-p_2,p_3,-p_3,p_4,-p_4,q_1,-q_1,q_2,-q_2)$,\\
\noindent $Y=(y_1,y_2,\cdots,y_{12})=$
\noindent $(r_1,r_2,-r_1,-r_2,r_3,r_4,-r_3,-r_4,s_1,s_2,-s_1,-s_2)$,
\\

\noindent Example:\\
$p_1=2$, $p_2=5$, $p_3=7$, $p_4=6$, $q_1=5$, $q_2=3$, $r_1=7$, $r_2=8$,
$r_3=9$, $r_4=9$, $s_1=10$, $s_2=6$, $A=-593$, $B=1129$,\\
\noindent Solution:\\
\noindent $(-1893)^\mu+(-7002)^\mu+9796^\mu+(-901)^\mu+2566^\mu+(-13240)^\mu+12111^\mu+(-1437)^\mu=
2 \cdot [(-285)^\mu+(-9203)^\mu+9317^\mu+171^\mu]$; $\mu=1,3$
\end{exam}

\begin{exam} $a^\mu+b^\mu+c^\mu+ \cdots +m^\mu+n^\mu+o^\mu=5 \cdot (p^\mu+q^\mu+r^\mu)$; $\mu=1,3$.\\

\noindent Trivial parametric solutions (case $2$):\\
\noindent $X=(x_1,x_2, \cdots,
x_{18})=$

\noindent $(p_1,-p_1,p_2,-p_2,p_3,-p_3,p_4,-p_4,p_5,-p_5,p_6,-p_6,p_7,-p_7,0,q_1,-q_1,0)$,\\
\noindent $Y=(y_1,y_2, \cdots,y_{18})=$\\
\noindent $(r_1,r_2,-r_1,-r_2,r_3,r_4,-r_3,-r_4,r_5,r_6,-r_5,-r_6,r_7,0,-r_7,s_1,0,-s_1)$,\\

\noindent Example:\\
$p_1=1$, $p_2=3$, $p_3=5$, $p_4=4$, $p_5=1$, $p_6=2$, $p_7=3$, $q_1=2$, $r_1=2$, $r_2=3$, $r_3=5$, $r_4=1$, $r_5=6$, $r_6=6$,
, $r_7=7$, $s_1=3$,
$A=-19$, $B=-253$,\\
\noindent Solution:\\
\noindent $(-291)^\mu+196^\mu+(-721)^\mu+816^\mu+(-1360)^\mu+1246^\mu+(-917)^\mu+1031^\mu
+(-367)^\mu+139^\mu+(-392)^\mu+
620^\mu+(-892)^\mu+759^\mu+133^\mu=
5 \cdot [(-563)^\mu+506^\mu+57^\mu]$; $\mu=1,3$
\end{exam}

\begin{exam} $a^\mu+b^\mu+c^\mu+d^\mu=2 \cdot (e^\mu+f^\mu)$; $\mu=1,3$.\\

\noindent Trivial parametric solutions(case $5$):\\
\noindent $X=(u,u,v,v,u,v)$,\\
\noindent $Y=(-v,1,-1,v,-u,u)$,\\

\noindent nontrivial parametric solution:\\
\noindent $a=-u^2-uv+uv^3+2u^4-vu^2+v^3-v^4$,\\
\noindent $b=-u^2v^2-uv-uv^3+2u^4+2u^3v-u^2v-v^2+v^3+2u^3-2uv^2$,\\
\noindent $c=-v^4+2u^2v^2-v^3-2u^3-u^2-uv-uv^3+2u^3v+u^2v+2uv^2$,\\
\noindent $d=u^2v^2-3uv^3+4u^3v-uv-v^2+vu^2-v^3$,\\
\noindent $e=u^2v^2-u^2-u^3-uv-2uv^3+3u^3v+uv^2$,\\
\noindent $f=-uv-v^2-v^4+u^3v+u^3-uv^2+2u^4$.\\

\noindent Example:\\
\noindent $u=2$, $ v=1$, or $u=10$, $v=7$, or $u=100$, $v=11$, or $u=100$, $v=1$,\\
\noindent Solutions:\\
\noindent $1^\mu+4^\mu+5^\mu+8^\mu=2 \cdot (2^\mu+7^\mu)$; $\mu=1,3$.\\
\noindent $201^\mu+257^\mu+168^\mu+224^\mu=2 \cdot (180^\mu+245^\mu)$; $\mu=1,3$.\\
\noindent $900895^\mu+1002355^\mu+100874^\mu+202334^\mu=2 \cdot (148400^\mu+954829^\mu)$; $\mu=1,3$.\\
\noindent $1^\mu+201^\mu+10000^\mu+10200^\mu=2 \cdot (100^\mu+10101^\mu)$; $\mu=1,3$.\\
\end{exam}

\begin{exam} $a^\mu+b^\mu+c^\mu+d^\mu+e^\mu+f^\mu=2 \cdot (g^\mu+h^\mu+i^\mu)$; $\mu=1,3$.\\

\noindent Trivial parametric solutions (case: $4$, $1$):\\
\noindent $X=(1,v,u,v,-u,-1,v,u,-u)$\\
\noindent $Y=(v,v,u,1,-1,-u,-u,v,u)$,\\

\noindent nontrivial parametric solution:\\
\noindent $a=-3u^3+v^4-v+u+u^2+vu^2-uv+2v^2u-u^3v+2v^3u-2v^2u^2$,\\
\noindent $b= -4u^3v+4v^3u$,\\
\noindent $c=-4u^3+4v^2u^2$,\\
\noindent $d=-v^4+v-u^3-u^2-u+uv-3u^3v+2u^2v^2+2v^3u-u^2v+2v^2u$,\\
\noindent $e=3u^4-v-v^3-v^2+u+uv-2v^2u^2+v^2u+v^3u-2vu^3-2v^2u+2u^2v$,\\
\noindent $f=4u^3+u^4+v-u+v^2+v^3-3v^2u-2vu^2-uv-uv^3-2u^2v^2+2u^3v$,\\
\noindent $g=u^4+u^3+u^2-v^3-v^4-v^2-u^3v+v^3u+u^2v-uv^2$,\\
\noindent $h=-3u^4+u^2+u^3+v^2+v^3+v^4-2uv+u^3v+v^3u-v^2u-u^2v$,\\
\noindent $i=2u^4-2u^3-2u^2+2uv+2v^3u+2v^2u-4vu^3$.\\

\noindent Example:\\
\noindent $u=2$, $v=1$, or $u=2$, $v=10$, or $u=7$, $v=100$,\\
\noindent Solutions:\\
\noindent $(-8)^\mu+(-8)^\mu+(-16)^\mu+(-8)^\mu+11^\mu+13^\mu=2 \cdot [7^\mu+(-11)^\mu+(-4)^\mu]$; $\mu=1,3$.
\\
\noindent $563^\mu+320^\mu+64^\mu+(-211)^\mu+(-5)^\mu+(-91)^\mu=2 \cdot [(-388)^\mu+536^\mu+172^\mu];$ $\mu=1,3$.
\\
\noindent $173777^\mu+42800^\mu+2996^\mu+(-130549)^\mu+7510^\mu+(-10934)^\mu=2 \cdot [(-144557)^\mu+165839^\mu+21518^\mu]$; $\mu=1,3$.
\end{exam}

\begin{exam} $a^\mu+b^\mu+c^\mu=d^\mu+e^\mu+f^\mu$; $\mu=1,3$.\\

\noindent Trivial parametric solutions (case $2$):\\
\noindent $X=(u,v,1,u,v,1)$,\\
\noindent $Y=(v,-1,-u,-1,-u,v)$,\\

\noindent nontrivial parametric solution:\\
\noindent $a=-u^3-v^3+vu^3-v^3u-2uv +v^2u-vu^2+u^2-v^2$,\\
\noindent $b=-uv^3-vu^2+uv+v^2u^2+v^3-u^2v+u+u^2+v^2u+v$,\\
\noindent $c=-uv^2-v+u+vu^2+v^2-u^3v+v^2u+u^3+v^2u^2+uv$,\\
\noindent $d=-u^2v^2-uv-u^3+u^2+vu^3+2uv^2-u^2v+v^2+u+u^2+v$,\\
\noindent $e=-uv^3-v^2-vu^2+2uv+2u^2v^2+v^3-u^3v+uv^2+u^2+u^3$,\\
\noindent $f=-uv^2-v-u^2+u+u^2v^2-v^3-uv-v^3u$.\\

\noindent Example:\\
\noindent $u=10$, $v=3$,\\
\noindent Solution:\\
\noindent $762^\mu+145^\mu+251^\mu=195^\mu+377^\mu+586^\mu$; $\mu=1,3$.\\
\end{exam}

\begin{exam} $a^\mu+b^\mu+c^\mu+d^\mu+e^\mu+f^\mu+g^\mu+h^\mu+i^\mu=3 \cdot (j^\mu+k^\mu+l^\mu)$; $\mu=1,3$.\\

\noindent Trivial parametric solutions (case: $1$, $2$):\\
\noindent $X=(x_1,x_2,\cdots,x_{12})=$
\noindent $(u,v,-u,t,-s,-t,0,s,-v,t,0,-t)$,\\
\noindent $Y=(y_1,y_2,\cdots,y_{12})=$
\noindent $(v,t,-s,u,-t,s,-v,-u,0,0,t,-t)$,\\

\noindent Example $1$:\\
$u=1$, $v=2$, $s=3$, $t=10$,\\
\noindent Solution:\\
\noindent $1168^\mu+7346^\mu+(-2003)^\mu+(-4185)^\mu+(-6844)^\mu+7525^\mu+(-1670)^\mu+(-2341)^\mu+
1004^\mu=
3 \cdot [(-5020)^\mu+8350^\mu+(-3330)^\mu]$; $\mu=1,3$.
\\

\noindent Example $2$:\\
\noindent $u=1$, $v=10$, $s=20$, $ t=30$,\\
\noindent Solution:\\
\noindent $83515^\mu+208720^\mu+(-173005)^\mu+(-170301)^\mu+
(-148970)^\mu+358230^\mu+(-89490)^\mu+(-128449)^\mu+59750^\mu=3 \cdot [(-179250)^\mu+268470^\mu+(-89220)^\mu]$; $\mu=1,3$.
\end{exam}

\section{Producing another solution of the DE ${\displaystyle \sum_{i=1}^n x_{i} ^3=k \cdot \sum_{j=1}^\frac{n}{k} y_{j}^3 }$, when a solution is given.}

Let $X=(x_1,x_2, \cdots ,x_n ,y_1,y_2, \cdots ,y_\frac{n}{k})$ be a primitive solution for the DE

\begin{equation}\label{97}
\sum_{i=1}^n x_{i} ^3=k \cdot \sum_{j=1}^\frac{n}{k} y_{j}^3,
\end{equation}

\noindent under conditions that :

\noindent $x_1+x_2+ \cdots +x_n -ky_1-ky_2- \cdots -ky_\frac{n}{k}\neq 0$, and\\
\noindent $-2x_1^{2}-2x_2^{2}- \cdots -2x_n^{2} +2ky_1^{2}+2ky_2^{2}+ \cdots +2ky_\frac{n}{k}^{2}\neq 0$.\\

\noindent (These conditions will be necessary because we wish the other solution not to be trivial.)
\\

\noindent We try to find another solution by using the first solution.\\
\noindent Define two variables function:

\noindent $F(x,y):=(2x_1x+y)^3+(2x_2x+y)^3+ \cdots +(2x_nx+y)^3$

\noindent $-k(2y_1x+y)^3-k(2y_2x+y)^3- \cdots -k(2y_\frac{n}{k}x+y)^3$.
\\

\noindent (Also, we may define: $F(x,y):=(tx_1x+y)^3+(tx_2x+y)^3+ \cdots +(tx_nx+y)^3$

\noindent $-k(ty_1x+y)^3-k(ty_2x+y)^3- \cdots -k(ty_\frac{n}{k}x+y)^3$, where $t$ is an arbitrary integer.)
\\

\noindent It is clear that if $F(x,y)=0$, then\\
\noindent $(2x_1x+y,2x_2x+y, \cdots ,2x_nx+y,2y_1x+y,2y_2x+y, \cdots ,2y_\frac{n}{k}x+y)$,\\
\noindent is another solution for the DE \eqref{97}.
\noindent We have

\noindent $F(x,y)=6xy[(2x_1^2+2x_2^2+ \cdots +2x_n^2-2ky_1^2- \cdots -2ky_\frac{n}{k}^2)x+
\\
(x_1+x_2+ \cdots +x_n-ky_1-ky_2- \cdots -ky_\frac{n}{k})y].$
\\

\noindent Therefore if we define\\
\noindent $x:=x_1+x_2+ \cdots +x_n-ky_1-ky_2- \cdots -ky_\frac{n}{k}\neq0$,

\noindent and

\noindent $y:=-2x_1^2-2x_2^2- \cdots -2x_n^2+2ky_1^2+ \cdots +2ky_\frac{n}{k}^2\neq0$,\\
\noindent then we get $F(x,y)=0$, which gives rise to another solution for the DE \eqref{97}:
\\

\noindent $x'_1=2x_1x+y=2x_1(x_1+x_2+ \cdots +x_n-ky_1-ky_2- \cdots -ky_\frac{n}{k})+$

\noindent $(-2x_1^2-2x_2^2- \cdots -2x_n^2+2ky_1^2+ \cdots +2ky_\frac{n}{k}^2),$
\\

\noindent $x'_2=2x_2x+y=2x_2(x_1+x_2+ \cdots +x_n-ky_1-ky_2- \cdots -ky_\frac{n}{k})+$

\noindent $(-2x_1^2-2x_2^2- \cdots -2x_n^2+2ky_1^2+ \cdots +2ky_\frac{n}{k}^2),$

\vdots

\noindent $x'_n=2x_nx+y=2x_2(x_1+x_2+ \cdots +x_n-ky_1-ky_2- \cdots -ky_\frac{n}{k})+$

\noindent $(-2x_1^2-2x_2^2- \cdots -2x_n^2+2ky_1^2+ \cdots +2ky_\frac{n}{k}^2),$
\\

\noindent $y'_1=2y_1x+y=2y_1(x_1+x_2+ \cdots +x_n-ky_1-ky_2- \cdots -ky_\frac{n}{k})+$

\noindent $(-2x_1^2-2x_2^2- \cdots -2x_n^2+2ky_1^2+ \cdots +2ky_\frac{n}{k}^2),$
\\

\noindent $y'_2=2y_2x+y=2y_2(x_1+x_2+ \cdots +x_n-ky_1-ky_2- \cdots -ky_\frac{n}{k})+$

\noindent$(-2x_1^2-2x_2^2- \cdots -2x_n^2+2ky_1^2+ \cdots +2ky_\frac{n}{k}^2),$
\\

\vdots

\noindent $y'_\frac{n}{k}=2y_\frac{n}{k}x+y=2y_2(x_1+x_2+ \cdots +x_n-ky_1-ky_2- \cdots -ky_\frac{n}{k})+$

\noindent $(-2x_1^2-2x_2^2- \cdots -2x_n^2+2ky_1^2+ \cdots +2ky_\frac{n}{k}^2)$.
\\

\noindent By continuing this method for the new obtained solution, we get infinitely many solutions for the DE \eqref{97}. This means that we may get infinitely many nontrivial solutions for the DE \eqref{97} by using a given solution.\\
If the primitive solution $X$ be a parametric solution, we get another parametric solution for the DE \eqref{97}.
\begin{exam} $a^3+b^3+c^3=A^3+B^3+C^3$\\

\noindent Let $(a,b,c, A, B,C)$ be a solution for the above DE.
\noindent We find another solution by using the first solution.\\

\noindent Example $1$: $(a,b,c,A,B,C)=(1,155,209,-41,227,107)$,\\
\noindent Another solution:\\
\noindent $125^3+169^3+250^3=62^3+277^3+97^3$.
\\

\noindent Example $2$: $(a,b,c,A,B,C)=(-62,169,250,-125,277,97),$\\
\noindent Another solution:\\
\noindent $81^3+12555^3+16929^3=(-3321)^3+18387^3+8667^3.$\\
\end{exam}

\section{The DE ${\displaystyle \sum_{i=1}^n p_{i} \cdot x_{i}^{a_i}=\sum_{j=1}^m q_{j} \cdot y_{j}^{b_j}}$;}
In this section, we prove the second main theorem. Let the above DE be in the form:\\
\noindent $ x_{1}^{a_1}+ p_{2}x_{2}^ {a_2}+ p_{3}x_{3}^{a_3}+ \cdots + p_{n}x_{n}^{a_n}=q_{1}y_{1}^{b_1}+ \cdots +q_{m}y_{m}^{b_m}$,\\
\noindent where $ (a_{1}, a_{2} \cdots a_{n} \cdot b_{1} \cdots b_{m})=1 $.\\

\noindent Then we have:
\begin{equation}\label{7}
x_{1}^ {a_1}=q_{1}y_{1}^{b_1}+ \cdots +q_{m}y_{m}^{b_m}- p_{2}x_{2}^{a_2}- p_{3}x_{3}^{a_3}- \cdots - p_{n}x_{n}^{a_n}.
\end{equation}

\noindent Define:\\ $m:=q_{1}s_{1}^{b_1}+ \cdots +q_{m}s_{m}^{b_m}- p_{2}t_{2}^{a_2}- p_{3}t_{3}^{a_3}- \cdots - p_{n}t_{n}^{a_n}$,\\
\noindent where $t_{i}$ and $ s_{i} $ are arbitrary integers. Now we introduce the our parametric solution:
\\

\noindent $y_{1}=s_{1} \cdot m^\frac{k}{b_{1}}$,
\\
$y_{2}=s_{2} \cdot m^\frac{k}{b_{2}}$,
\\
$\vdots$
\\
$y_{m}=s_{m} \cdot m^\frac{k}{b_{m}}$,
\\
$ x_{1}=m^\frac{k+1}{a_{1}}$,
\\
$ x_{2}=t_{1} \cdot m^\frac{k}{a_{2}}$,
\\
$\vdots$
\\
$ x_{n}=t_{n} \cdot m^\frac{k}{a_{n}}$,
\\

\noindent where $k$ is a natural number such that
\\

\noindent $k\equiv 0 \pmod {b_{1}}$,
\\
$k\equiv 0 \pmod {b_{2}}$,
\\
$\vdots$
\\
$k\equiv 0 \pmod {b_{m}}$,
\\
$k\equiv 0 \pmod {a_{2}}$,
\\
$\vdots$
\\
$k\equiv 0 \pmod {a_{n}}$,
\\
$k\equiv -1 \pmod {a_{1}}$.
\\

\noindent From the Chinese remainder theorem, we know that there exists a solution for $k$, since $(a_{1},a_{2} \cdots a_{n} \cdot b_{1} \cdots b_{m})=1$, that is, the all exponents used in $x_{i}$ and $y_{j}$
are natural numbers.
We claim that this is a parametric solution for the DE \eqref{7}:
\\

\noindent $q_{1}y_{1}^{b_1}+ \cdots +q_{m}y_{m}^{b_m}- p_{2}x_{2}^{a_2}- p_{3}x_{3}^{a_3}- \cdots - p_{n}x_{n}^{a_n}=$\\

\noindent $q_{1} \cdot (s_{1}m^\frac{k}{{b_1}})^{b_1}+q_{2} \cdot (s_{2}m^\frac{k}{b_{2}})^{b_2}+ \cdots +q_{m} \cdot (s_{m}m^\frac{k}{b_{m}})^{b_m} -p_{2}(t_{2} \cdot m^\frac{k}{a_{2}})^{a_2}-$\\

\noindent $ \cdots -p_{n}(t_{n} \cdot m^\frac{k}{a_{n}})^{a_n}=$\\

\noindent $q_{1} \cdot s_{1}^{b_1} \cdot m^k+q_{2} \cdot s_{2}^{b_2} \cdot m^k+ \cdots +q_{m} \cdot s_{m}^{b_m} \cdot m^k-p_{2} \cdot t_{2}^{a_2} \cdot m^k- \cdots -p_{n} \cdot t_{n}^{a_n} \cdot m^k=$\\

\noindent $m^k(q_{1} \cdot s_{1}^{b_1}+
q_{2} \cdot s_{2}^{b_2}+ \cdots +q_{m} \cdot s_{m}^{b_m}-p_{2} \cdot t_{2}^{a_2}- \cdots -p_{n} \cdot t_{n}^{a_n})=m^k \cdot m=$
\\

\noindent $m^{k+1}=(m^\frac{k+1}{a})^a=x_{1}^a.$
\\

\noindent Now the proof of the second main theorem is completed. Since $s_{i}$, and $ t_{i}$ were arbitrary, we also obtain infinitely many solutions for the above DE.
\begin{exam}
The DE $x_{1}^5+x_{2}^6=y_{1}^7+y_{2}^8+y_{3}^9$, has infinitely many solutions in integers.
\\
Since we have: $x_{1}^5=y_{1}^7+y_{2}^8+y_{3}^9-x_{2}^6 $ and $(5,6 \cdot 7 \cdot 8 \cdot 9)=1$, then by using the previous theorem, we conclude that the aforementioned DE has infinitely many solutions in integers.
Put: $m:=r^7+s^8+t^9-w^6$, where $r$, $s$, $t$ and $w$ are arbitrary integers. As an example, if we let: $(r,s,t,w)=(3,2,1,3)$ and $k=504$, then we get:
\\
$m=1715$,
\\
$y_{1}=r \cdot m^\frac{k}{7}=3 \cdot 1715^{72}$,
\\
$y_{2}=s \cdot m^\frac{k}{8}=2 \cdot 1715^{63}$,
\\
$y_{3}=t \cdot m^\frac{k}{9}=1715^{56}$,
\\
$x_{1}=m^\frac{k+1}{5}=1715^{101}$,
\\
$x_{2}=w \cdot m^\frac{k}{6}=3 \cdot 1715^{84}$.
\\

\noindent Namely, we have:\\
$(1715^{101})^5+(3 \cdot 1715^{84})^6=(3 \cdot 1715^{72})^7+(2 \cdot 1715^{63})^8+(1715^{56})^9.$
\\

\noindent By letting $(r,s,t,w)=(3,2,2,2)$ and $k=504$, we get:\\
$(2891^{101})^5+(2 \cdot 2891^{84})^6=(3 \cdot 2891^{72})^7+(2 \cdot 2891^{63})^8+(2 \cdot 2891^{56})^9.$\\

\noindent By changing $(r,s,t,w)$, we obtain infinitely many solutions.
\end{exam}

\begin{exam}

The DE $x^5=6(y_{1}^7+y_{2}^7+y_{3}^7)$ has infinitely many solutions in integers. If we put: $m:=6r^7+6s^7+6t^7$, where $(r,s,t)=$ $(1,1,1)$ and $k=14$, we get:
\\
$m=18$,
\\
$y_{1}=r \cdot m^\frac{k}{7}=18^2$,
\\
$y_{2}=s \cdot m^\frac{k}{7}=18^2$,
\\
$y_{3}=t \cdot m^\frac{k}{7}=18^2$,
\\
$x=m^\frac{k+1}{5}=18^3$.
\\

\noindent Namely, we have: $(18^3)^5=6((18^2)^7+(18^2)^7+(18^2)^7)$.
\\

\noindent By letting: $(r,s,t)=(3,2,5)$ and $k=14$, we obtain:\\
$(482640^3)^5=6((3 \cdot 482640^2)^7+(2 \cdot 482640^2)^7+(5 \cdot 482640^2)^7)$.

\end{exam}

\begin{exam}
The DE $x^7=13y_{1}^{11}+ 11y_{2}^{13}$
has infinitely many solutions in integers. If we put: $m:=13r^{11}+11s^{13}$, where $(r,s,t)=(1,1,1)$ and $k=286$, we get:
\\
$m=24$,
\\
$y_{1}=r \cdot m^\frac{k}{11}=24^{26}$,
\\
$y_{2}=s \cdot m^\frac{k}{13}=24^{22}$,
\\
$x=m^\frac{k+1}{7}=24^{41}$.
\\

\noindent Namely, we have: $(24^{41})^7=13 \cdot (24^{26})^{11}+11 \cdot (24^{22})^{13}$.
\end{exam}

\begin{exam}
Is the DE $x_{1}^{5}+x_{2}^{5}=y_{1}^{6}+y_{2}^{6}$
solvable? Yes. To do this, let $ y_{1}=y_{2}$. Then we have: $x_{1}^{5}+x_{2}^{5}=2y_{1}^{6}$. We may let $x_{1}=2u$, $x_{2}=2v$. So we get: $y_{1}^{6}=2^4(u^5+v^5)$, that is solvable.
\end{exam}
\section{The DE $x^a+y^b=z^c$}
In this part, we solve the above DE, where $a$, $b$, $c$ are fixed natural numbers and $x$, $y$, $z$ are variables. We have three cases: $(a,b,c)=1$ , $(a,b,c)\geq3$ or $(a,b,c)=2$.
If $(a,b,c)=1$, by using the second main theorem, we proved that the DE has infinitely many solutions in integers.
If $(a,b,c)\geq3$, the Fermat last theorem says that the DE dose not have any solutions in integers. Then it suffices to study the case $(a,b,c)=2$. However, we know that the Diophantine equations $x^4+y^4=z^2$ and $x^4-y^4=z^2$ have not any solutions in integers. So it suffices only to study the case where at most one of $a$, $b$, $c$ is divisible by $4$. In the sequel, we study this case by proving several theorems.
\begin{theorem}
The DE $x^2+y^{2B} =z^{2C}$, where $B$ and $C$ are both odd and $(B,C)=1$, has infinitely many solutions in integers.
\end{theorem}
\begin{proof}
: we have $x^2+(y^{B})^2 =(z^{C})^2$, then we get:\\
$x=m^2-n^2$,
\\
$ y^{B}=2mn$,
\\
$z^C=m^2+n^2$.
\\
We may suppose that $m=2^{B-1} \cdot t_{1}^{B}$, $ n=t_{2}^{B}$. By plugging these into the equations, we get :
\\
$y=2t_{1}t_{2}$,
\\
$x=2^{2(B-1)} \cdot t_{1}^{2B}-t_{2}^{2B}$,
\\
$z^{C}=2^{2(B-1)} \cdot t_{1}^{2B}+t_{2}^{2B}$.
\\

\noindent Since $(C,2B)=1$, the DE $z^{C}=2^{2(B-1)} \cdot t_{1}^{2B}+t_{2}^{2B}$ is solvable for $z$, $t_{1}$, $t_{2}$ from the second main theorem. Also $x$ and $y$ are computed from
$y=2t_{1}t_{2}$,
$x=2^{2(B-1)} \cdot t_{1}^{2B}-t_{2}^{2B}$, as well.
The proof is completed.
\end{proof}
\begin{theorem}
The DE $x^{2A}+y^{2B} =z^{2C}$, where $A$, $B$, $C$ are all odd and $(A,C)=1$ and $(AC,B)=1$ has infinitely many solutions in integers.
\end{theorem}
We note that in this case $(a,b,c)=(2A,2B,2C)=2$ and none of $a=2A$, $b=2B$ and $c=2C$ is divisible by $4$. We wish to solve $x^a+y^b=z^c$ in the case of $(a,b,c)=2$ and at most one of $a$, $b$, $c$ is divisible by $4$, this theorem is one of the desired cases where $(a,b,c)=2$ and none of $a$, $b$ and $c$ is divisable by $4$.
\begin{proof}
: We know that the DE $X^A+Y^B=Z^C$ with the condition $(AC,B)=1$ is solvable and its solution is:
\\
$m:=r^C-s^A$,
\\
$Z=m^{\frac{k}{C}} \cdot r$,
\\
$X=m^{\frac{k}{A}} \cdot s$,
\\
$Y=m^{\frac{k+1}{B}}$,
\\
and
\\
$k\equiv0 \pmod {AC}$,
\\
$k\equiv-1 \pmod{ B}$,
\\
where $r$ and $s$ are arbitrary integers.
Now, if we obtain a solution for the DE $X^A+Y^B=Z^C$ from the solution just introduced, where $X$, $Y$ and $Z$ are squares, then we get a solution for the DE $x^{2A}+y^{2B} =z^{2C}$ by putting $X=x^2$, $Y=y^2$ and $Z=z^2$.
From the above solution for $X$, $Y$ and $Z$ , we see that if $m$, $r$ and $s$ be squares, then $X$, $Y$ and $Z$ will be squares, as well.
Then since $m$, $r$ and $s$ are related together with $m:= r^C-s^A$, we see that it suffices to solve the DE $M^2=R^{2C}-S^{2A}$, where we set:
\\
$m=M^2$,
\\
$r=R^2$,
\\
$s=S^2$.
\\
(we wanted $m$, $r$ and $s$ to be squares).
Fortunately the DE $M^2=R^{2C}-S^{2A}$ is solvable from the previous theorem due to $(A,C)=1$. Then we obtain $m$, $r$, $s$, and in the end we get $x$, $y$, $z$, and the proof is completed.
\end{proof}
\begin{theorem}
The DE $x^{4A}+y^{2B} =z^{2C}$, where $B$ and $C$ are both odd and $(A,C)=1$ and $(AC,B)=1$ has infinitely many solutions in integers.
\end{theorem}
\begin{proof}
: We know that the DE $X^{2A}+Y^B=Z^C$ with the condition $(2AC,B)=1$ is solvable and its solution is:
\\
$m:=r^C-s^{2A}$,
\\
$Z=m^{\frac{k}{C}} \cdot r$,
\\
$X=m^{\frac{k}{2A}} \cdot s$,
\\
$Y=m^{\frac{k+1}{B}}$,
\\
and
\\
$k\equiv0 \pmod {C}$,
\\
$k\equiv0 \pmod {2A}$,
\\
$k\equiv-1 \pmod {B}$,
\\
where $r$ and $s$ are arbitrary integers.
Now, if we obtain a solution for the DE $X^{2A}+Y^B=Z^C$ from the solution just introduced, where $X$, $Y$ and $Z$ are squares, next we get a solution for the DE $x^{4A}+y^{2B} =z^{2C}$ by putting $x=X^2$, $y=Y^2$ and $z=Z^2 $.
From the above solution for $X$, $Y$, and $Z$, we see that if $m$, $r$ and $s$ be squares, then $X$, $Y$, and $Z$ will be squares, as well.
Then since m, r and s are related together with $m:=r^C-s^{2A}$, we see that it suffices to solve the DE $M^2=R^{2C}-S^{4A}$, where we set:
\\
$m=M^2$,
\\
$r=R^2$,
\\
$s=S^2$.
\\
(We wanted $m$, $r$ and $s$ to be squares).
Fortunately the DE $M^2=R^{2C}-S^{4A}$ is solvable from the previous theorems:
we have $M^2+(S^{2A})^2 =(R^{C})^2$, and then we get:
\\
$M=t_{1}^2-t_{2}^2$,
\\
$ S^{2A}=2t_{1}t_{2}$,
\\
$R^C=t_{1}^2+t_{2}^2$.
\\
We may suppose that $t_{1}=2^{2A-1} \cdot p^{2A}$, $ t_{2}=q^{2A}$. By substituting these in the above equations, we get :
\\
$S=2pq$,
\\
$M=2^{2(2A-1)} \cdot p^{4A}-q^{4A}$,
\\
$R^{C}=2^{2(2A-1)} \cdot p^{4A}+q^{4A}$.
\\
Since $(C,4A)=1$, the DE $R^{C}=2^{2(2A-1)} \cdot p^{4A}+q^{4A}$ is solvable for $R$, $p$ and $q$, from the second main theorem. And $M$, and $S$ are computed from
$S=2pq$,
$M=2^{2(2A-1)} \cdot p^{4A}-q^{4A}$, as well,
the proof is completed.
\end{proof}
We note that the previous theorem is the case that in the DE $x^a+y^b=z^c$, exactly one of $a$ or $b$ is divisible by $4$, and $(a,b,c)=2$.
\begin{exam}
We wish to solve the DE $x^{6}+y^{10}=z^{14}$.
We know that the DE $x^{3}+y^{5}=z^{7}$ is solvable and its solution is:
\\
$m:=r^7-s^3$,
\\
$z=m^{\frac{k}{7}} \cdot r$,
\\
$x=m^{\frac{k}{3}} \cdot s$,
\\
$y=m^{\frac{k+1}{5}}$,
\\
and
\\
$k\equiv0 \pmod {21}$,
\\
$k\equiv-1 \pmod {5}$,
\\
where $r$, and $s$ are arbitrary integers.\\
\noindent Now, if we obtain a solution for the DE $x^3+y^5=z^7$ from the solution just introduced, where $x$, $y$ and $z$ are squares, then we get a solution for the DE $x^6+y^{10} =z^{14}$.
From the above solution for $x$, $y$ and $z$, we see that if $m$, $r$ and $s$ be squares, then $x$, $y$ and $z$ will be squares, as well.
Since $m:=r^7-s^3$, we see that it suffices to solve the DE $M^2=R^{14}-S^{6}$, where:
\\
$m=M^2$,
\\
$r=R^2$,
\\
$s=S^2$.
Fortunately the DE $M^2=R^{14}-S^6$ is solvable from the previous theorems:
we have $M^2+S^6 =R^{14}$, and then we get:
\\
$M=t_{1}^2-t_{2}^2$,
\\
$ S^3=2t_{1}t_{2}$,
\\
$R^7=t_{1}^2+t_{2}^2$.
\\
We may suppose that $t_{1}=4p^3$,$ t_{2}=q^3$. By substituting these in the above equations, we get :
\\
$S=2pq$,
\\
$M=16p^6-q^6$,
\\
$R^7=16p^6+q^6$.
\\
Since $(7,6)=1$, the DE $R^7=16p^6+q^6$ is solvable for $R$, $p$ and $q$ from the previous theorems:
\\
$M':=16u^6+v^6$,
\\
$p=M'^{\frac{k}{6} } \cdot u$,
\\
$q=M'^{\frac{k}{6}} \cdot v$,
\\
$R=M'^{\frac{k+1}{7}}$,
\\
and
\\
$K\equiv0 \pmod {6}$ ,
\\
$k\equiv-1\pmod {7}$.
\\

\noindent By putting $u=v=1$, $k=6$, we get:\\
\noindent $(15^{28} \cdot 17^{170} \cdot 2)^{6}+(15^{17} \cdot 17^{102})^{10}=(15^{12} \cdot 17^{73})^{14}$.
\end{exam}
\begin{theorem}
The DE $x^2+y^2 =z^{2^n \cdot c}$, where $c$ is odd and $n\geq2$, is solvable.
\end {theorem}
\begin{proof}
: We try to solve the DE by using induction. If $n=2$, we have\\
$x^2+y^2=z^{4c}$,
then
\\
$x=m^2-n^2$,
\\
$y=2mn$,
\\
$z^{2c}=m^2+n^2$.
\\
Again we put:
\\
$n=2t_{1}t_{2}$,
\\
$m=2t_{1}^2-t_{2}^2$,
\\
$z^c=t_{1}^2+t_{2}^2$.
\\
Since $(2,c)=1$, the DE $z^c=t_{1}^2+t_{2}^2$ is solvable, then the main DE is solvable as well, and the proof for $n=2$ is complete.
Next, suppose the above DE is solvable for $n$, then we solve it for $n+1$. We have:
\\
$x^2+y^2 =z^{2^{n+1} \cdot c}$,
then
\\
$x=m^2-n^2$,
\\
$y=2mn$,
\\
$z^{2^{n} \cdot c}=m^2+n^2$.
\\
Since the DE $z^{2^{n} \cdot c}=m^2+n^2$, is solvable for $m$, $n$, and $z$ by our assumption, then the DE $x^2+y^2 =z^{2^{n+1} \cdot c}$, is solvable as well, and the proof is completed.
\end{proof}
\begin{rem}
If in the previous theorem, we take $c=1$, we can solve it by another beautiful method. We start from the identity
\\
$(m^2-n^2)^2+(2mn)^2=(m^2+n^2)^2$.\\

\noindent We see that: if in the above identity $m^2+n^2$, to be square, then we get a solution for the DE $x^2+y^2=z^4$. As an example, we can get $m$ and $n$ from the Pythagorean triples. For $m=4$, $n=3$, we obtain:\\
$7^2+24^2=(5^2)^{2}=(3^2+4^2)^{2}=5^4$.\\

\noindent By letting $n=7$, $m=24$, we obtain:
$572^2+336^2=(7^2+24^2)^{2}=(5^4)^{2}=5^8$.\\

\noindent By continuing
in this way, if we put $m=572$, $n=336$, we get a solution for the DE $x^2+y^2=z^{16}$.
If we change $m$, $n$, we obtain another identity. If $m=120$, $n=119$, we have :
$239^2+28560^2=13^8$.
It is clear that by using this method we can find infinitely many solutions for the DE
$x^2+y^2 =z^{2^n \cdot c}$ with $(x,y,z)=1$, as mentioned in the Beukers theorem (see \cite{3}) in the introduction.
\end{rem}
In the end, it is clear that at some variables parametric solutions obtained for the Diophantine equations, we may get one variable parametric solutions for each case of the above Diophantine equations by fixing the other variables.

\begin{center}\textbf{Acknowledgements}
\end{center}
The authors would like to express their hearty thanks to the
anonymous referee for a careful reading of the paper and for many
careful comments and remarks which improved its quality.




\end{document}